\documentclass[11pt,twoside]{amsart}

\usepackage{amssymb}
\usepackage{verbatim}

\numberwithin{equation}{section}
\newtheorem{theorem}{Theorem}[section]
\newtheorem{lemma}[theorem]{Lemma}

\newtheorem{corollary}[theorem]{Corollary}
\newtheorem{conjecture}[theorem]{Conjecture}

\theoremstyle{definition}

\newtheorem{remark}[theorem]{Remark}
\newtheorem{example}[theorem]{Example}

\newcommand{\popo}{\mathbb{P}^1 \times \mathbb{P}^1}

\newcommand{\pr}{\mathbb{P}}

\begin{document}


\title[Asymptotic resurgences]{
Asymptotic resurgences for ideals of positive dimensional 
subschemes of projective space}

\author{Elena Guardo}
\address{Dipartimento di Matematica e Informatica\\
Viale A. Doria, 6 - 95100 - Catania, Italy}
\email{guardo@dmi.unict.it}
\urladdr{http://www.dmi.unict.it/$\sim$guardo/}

\author{Brian Harbourne}
\address{Department of Mathematics \\
University of Nebraska--Lincoln\\
Lincoln, NE 68588-0130, USA }
\email{bharbour@math.unl.edu}
\urladdr{http://www.math.unl.edu/$\sim$bharbourne1/}

\author{Adam Van Tuyl}
\address{Department of Mathematical Sciences \\
Lakehead University \\
Thunder Bay, ON P7B 5E1, Canada}
\email{avantuyl@lakeheadu.ca}
\urladdr{http://flash.lakeheadu.ca/$\sim$avantuyl/}

\thanks{February 29, 2012}

\begin{abstract}
Recent work of Ein-Lazarsfeld-Smith and Hochster-Huneke 
raised the problem of determining which symbolic powers of an ideal
are contained in a given ordinary power of the ideal.
Bocci-Harbourne defined a quantity called the resurgence to address this problem
for homogeneous ideals in polynomial rings, with
a focus on zero dimensional 
subschemes of projective space; the methods and results
obtained there have much less to say about higher dimensional subschemes. 
Here we take the first steps toward extending this work to higher dimensional
subschemes. We introduce new asymptotic
versions of the resurgence and obtain upper and lower bounds
on them for ideals of smooth subschemes, 
generalizing what is done in \cite{BH}. We apply these bounds
to ideals of unions of general lines in $\pr^N$.
We also pose a Nagata type conjecture for symbolic powers of ideals of
lines in $\pr^3$.
\end{abstract}

\maketitle



\section{Introduction}

Refinements of the groundbreaking
results of \cite{refELS, HaHu} have recently been given in \cite{BH, BH2, BCH, HH}. 
The resurgence, introduced in \cite{BH} as a way of quantifying
which symbolic powers of an ideal are contained in a given power of the ideal,
is especially useful for studying homogeneous ideals defining
zero dimensional subschemes of projective space.  However, except for
special cases such as projective cones over points or when the resurgence is equal to 1
(such as for complete intersections), there are no ideals of positive dimensional subschemes
for which the resurgence is known. In order to extend this work
to higher dimensional subschemes, we introduce asymptotic 
refinements of the resurgence better adapted to studying ideals $I$ defining 
higher dimensional subschemes of projective space, but as a tradeoff 
we need to require that the subschemes be smooth.
Our main result gives upper and lower bounds
on these asymptotic resurgences in terms of three numerical characters of $I$
(viz., the least degree $\alpha(I)$ and largest degree $\omega(I)$
in a minimal homogeneous set of generators of $I$,
and the asymptotic invariant $\gamma(I)=\lim_{m\to\infty}\alpha(I^{(m)})/m$,
where $I^{(m)}$ is the $m$th symbolic power of $I$). We apply this result to 
obtain explicit bounds in the case of ideals of unions of general lines in $\pr^N$ for $N\geq 3$,
in some cases reducing the exact determination of an asymptotic resurgence
to determining $\gamma(I)$.

Throughout this paper, we work over an algebraically closed field $k$,
of arbitrary characteristic. Let $(0)\subsetneq I\subsetneq (1)$ be a homogeneous
ideal in the homogeneous coordinate
ring $k[\pr^N]=k[x_0,\ldots,x_N]$ of $\pr^N$.  Then
\[\rho(I) = \sup\left\{m/r ~:~ I^{(m)} \not\subseteq I^r\right\}\]
is called the {\em resurgence} of $I$ \cite{BH}
(see section \ref{prelims} below for the definiton of the symbolic power $I^{(m)}$).
By definition, if $m$ and $r$ are positive integers such that $m/r>\rho(I)$, then
$I^{(m)}\subseteq I^r$. However, $\rho(I)$ is often very hard to compute.
Even showing that $\rho(I)$ is finite is not easy, but it follows from
\cite{refELS,HH} that $\rho(I)\le N$, and more generally that
$\rho(I)\le h_I$ where $h_I$ is the maximum height of an associated prime of $I$.

Lower bounds for $\rho(I)$ were given in \cite{BH}.
For each $1\leq h\leq N$, these bounds show that ideals $I$ with $h_I = h$ can be found where
$\rho(I)$ is arbitrarily close to $h$
(in particular, the supremum of the values of $\rho(I)$ over all homogeneous ideals
$I$ with $h_I = h$ is $h$; we note that
no examples are known for which $\rho(I)$ is equal to $h$ when $h > 1$).
In some situations, \cite{BH} gives an exact value for $\rho(I)$ by giving
an upper bound on $\rho(I)$ which coincides with the lower bound.  However,
these upper bounds require that $I$ define a zero-dimensional subscheme of
$\pr^N$ (or a cone over a finite set of points; see \cite[Proposition 2.5.1(b)]{BH}).
The only other cases of ideals of positive dimensional subschemes
where upper bounds other than $\rho(I)\leq h_I$ are known
occur for ideals with $I^{(r)}=I^r$ for all $r\geq 1$ (such as for complete intersections;
see section \ref{prelims}),
in which case the resurgence is 1.
Thus, much less is known about values of $\rho(I)$ when $\rho(I)$ is not known to be 1
and when the subscheme
that $I$ defines is positive dimensional but not a cone over a finite set of points.

\subsection{Discussion of results}

Given the difficulty in finding $\rho(I)$, it is worth looking
at variants of the resurgence. For a homogeneous ideal $(0)\subsetneq I\subsetneq k[\pr^N]$,
we introduce an {\em asymptotic resurgence}, which we define as
\[\rho_a(I)=\sup\{m/r: I^{(mt)}\not\subseteq I^{rt} \text{ for all }t\gg0\}.\]
We will also consider an additional asymptotic version of the resurgence.
We have 
\[\rho'_a(I)=\limsup_{t\to\infty}\rho(I,t),\]
where $\rho(I,t)=\sup\{m/r: I^{(m)}\not\subseteq I^{r}, m\geq t, r\geq t\}$.

\begin{remark}\label{introrem}
Clearly $m/r>\rho_a(I)$ implies $I^{(mt)}\subseteq I^{rt}$ for 
infinitely many $t>0$. Moreover, $m/r<\rho_a(I)$ implies
that $I^{(mt)}\not\subseteq I^{rt}$ for infinitely many $t>0$.
(If $m/r<\rho_a(I)$, then there is a pair $(m',r')$ such that
$m/r\leq m'/r'\leq \rho_a(I)$ with $I^{(m't)}\not\subseteq I^{r't}$ for all $t\gg0$,
and so $I^{(mm't)}\not\subseteq I^{r'mt}$ for all $t\gg0$.
But $r'm\leq m'r$ so $I^{rm't}\subseteq I^{r'mt}$, hence
$I^{(mm't)}\not\subseteq I^{rm't}$ for all $t\gg0$, so
$I^{(mt)}\not\subseteq I^{rt}$ for infinitely many $t>0$.)

Moreover, $m/r>\rho'_a(I)$ implies $I^{(m')}\subseteq I^{r'}$ for all $m'\gg0$, $r'\gg0$
such that $m'/r'\geq m/r$, while $m/r<\rho'_a(I)$ implies 
there are infinitely many pairs $(m',r')$ with $m/r\leq m'/r'\leq \rho'_a(I)$
such that $I^{(m')}\not\subseteq I^{r'}$.
\end{remark}

Our main result is Theorem \ref{mainThm}; it gives both upper and lower bounds on
these asymptotic versions of the resurgence for certain ideals.
Our lower bound is the same as that obtained in \cite{BH},
and applies to any nontrivial homogeneous ideal.
Our upper bound is closely related to that of \cite{BH}
but applies to the ideal of any smooth subscheme of $\pr^N$,
whereas the upper bound on the resurgence given
in \cite{BH} applies only to ideals defining zero-dimensional subschemes
(although the subschemes do not need to be smooth).

To state the result, given a homogeneous ideal $(0)\neq I\subseteq k[\pr^N]$,
we define $\alpha(I)$ to be the least degree
of a nonzero element of $I$, $\omega(I)$ to be the largest degree
in a minimal homogeneous set of generators of $I$,
and $\gamma(I)$ to be $\lim_{m\to\infty}\alpha(I^{(m)})/m$; see section \ref{prelims}
for basic facts about $\gamma(I)$.
We remark that whereas bound (1) in our result below is what was actually 
proved (but not stated) in \cite{BH}, bound (2)
generalizes the upper bound obtained in \cite{BH}.
Note by \cite[Lemma 2.3.3(a)]{BH} that $I^{(m)}\not\subseteq I^r$ if $m<r$ 
for a homogeneous ideal $(0)\neq I\neq (1)$, and hence if $I^{(m)}=I^m$ for all $m\geq 1$,
then $I^{(m)}=I^m\subseteq I^r$ if and only if $r\leq m$, so $\rho_a(I)=\rho'_a(I)=\rho(I)=1$; 
bound (3) generalizes this fact. This generalization applies when there is a $c$ such that
$I^{(cm)}=(I^{(c)})^m$ for all $m\geq 1$, which occurs, for example, 
whenever the symbolic power algebra
$\bigoplus_jI^{(j)}$ is Noetherian; see \cite[Proposition 2.1]{refSc} or \cite{refR}.

\begin{theorem}\label{mainThm}
Consider a homogeneous ideal $(0)\neq I\subsetneq k[\pr^N]$.
Let $h=\min(N,h_I)$ where $h_I$ is the maximum of the heights of 
the associated primes of $I$.
\begin{enumerate}
\item[(1)] We have $1\leq \alpha(I)/\gamma(I)\leq\rho_a(I)\leq\rho_a'(I)\leq\rho(I)\leq h$.
\item[(2)] If $I$ is the ideal of a (non-empty) smooth subscheme of $\pr^N$, then
$$\rho_a(I)\le \frac{\omega(I)}{\gamma(I)}\leq \frac{\operatorname{reg}(I)}{\gamma(I)}.$$
\item[(3)] If for some positive integer $c$ we have $I^{(cm)}=(I^{(c)})^m$ for all $m\geq 1$,
and if $I^{(c)}\subseteq I^b$ for some positive integer $b$, then
$$\rho'_a(I)\leq \frac{c}{b}.$$
\end{enumerate}
\end{theorem}

As an application of Theorem \ref{mainThm}(1,2), we obtain the following result:

\begin{corollary}\label{introcor1}
Let $I$ be the ideal of $s$ general lines in $\pr^N$ for $N\geq 3$, where 
$s=\binom{t+N}{N}/(t+1)$ for any integer $t\geq 0$ such that $s$ is an integer
(there are always infinitely many such $t$; for example, let $t=p-1$ for a prime $p>N$).
Then $\rho_a(I)=(t+1)/\gamma(I)$.
\end{corollary}

Here are examples demonstrating Theorem \ref{mainThm}(3).

\begin{example}
Let $I$ be the ideal of $n$ general points in $\pr^2$.
If $n=6$, we have $I^{(10m)}=(I^{(10)})^m$ \cite[end of Section 3]{HaHu} and $I^{(10)}\subseteq I^{8}$ 
\cite[Proposition 4.1]{BH2}
so $\rho'_a(I)\leq 10/8=5/4$ by Theorem \ref{mainThm}(3).
In fact, $\rho(I)=5/4$ by \cite[Proposition 4.1]{BH2},
so this also follows from Theorem \ref{mainThm}(1).
Moreover, $I^{(m)}\not\subseteq I^r$ if and only if $m<5r/4-5/12$ by \cite[Proposition 4.1]{BH2}.
It follows that $I^{(mt)}\not\subseteq I^{rt}$ for all $t\gg0$ whenever
$m/r<5/4$ so $5/4\leq \rho_a(I)$ and hence we have $\rho_a(I)=\rho'_a(I)=\rho(I)=5/4$.
Similarly, if $n=7$ we have $I^{(24m)}=(I^{(24)})^m$ 
\cite{HaHu} and $I^{(24)}\subseteq I^{21}$ 
\cite[Proposition 4.3]{BH2}
so $\rho'_a(I)\leq 24/21=8/7$ (here also we have 
$\rho_a(I)=\rho'_a(I)=\rho(I)=8/7$ since 
$I^{(m)}\not\subseteq I^r$ if $m<8r/7$ for $m>1$ by \cite[Proposition 4.3]{BH2}, so
it follows that $I^{(mt)}\not\subseteq I^{rt}$ for all $t\gg0$ whenever
$m/r<8/7$).
Finally, if $n=8$, then $I^{(102m)}=(I^{(102)})^m$ for all $m\geq 1$ by \cite{HaHu}
and $I^{(102)}\subseteq I^{72}$ by \cite[Proposition 4.4]{BH2}, so 
$\rho'_a(I)\leq 102/72=17/12$ (and again 
$\rho_a(I)=\rho'_a(I)=\rho(I)=17/12$, since  
$I^{(m)}\not\subseteq I^r$ if $m<17r/12-1/3$ for $m>1$ by 
\cite[Proposition 4.4]{BH2}).
\end{example}

It is in general an open problem to compute $\gamma(I)$. It is often
challenging just to find good lower bounds for $\gamma(I)$; see \cite{refEV, HaHu, refSch}.
Thus Corollary \ref{introcor1} shows that the problem of 
computing $\rho_a(I)$ is related to a significant open problem. Although it can 
be difficult to determine $\gamma(I)$ exactly, it is possible in principle to estimate it
to any desired precision; see section \ref{prelims}.
In some cases, however, such as for ideals of $s$
general lines in $\pr^N$ for small values of $s$, determining $\gamma(I)$ 
and $\rho(I)$ (and also $\rho_a(I)$ and $\rho'_a(I)$) is much easier:

\begin{theorem}\label{introthm2}
Let $I$ be the ideal of $s$ general lines in $\pr^N$ for $N\geq 2$ and
$s\leq (N+1)/2$. Then $\rho(I)=\rho'_a(I)=\rho_a(I)=\max(1,2\frac{s-1}{s})$. Moreover, if $2s<N+1$,
then $\gamma(I)=1$, while if $2s=N+1$, then $\gamma(I)=\frac{N+1}{N-1}$.
\end{theorem}

In section \ref{prelims} we recall basic facts that we will need for the proofs.
We give the proofs of Theorem \ref{mainThm} and Corollary \ref{introcor1} in section
\ref{tmt} and the proof of Theorem \ref{introthm2} in section \ref{monomialcases}.
We discuss the problem of computing $\gamma(I)$ in more detail in section \ref{gamma}.

\noindent
{\bf Acknowledgments.}  
This work was facilitated by the Shared Hierarchical
Academic Research Computing Network (SHARCNET:www.sharcnet.ca) and Compute/Calcul Canada.
Computer experiments carried out on {\tt CoCoA} \cite{C} and
{\it Macaulay2} \cite{Mt} were very helpful in guiding our research.
The second author's work on this project
was sponsored by the National Security Agency under Grant/Cooperative
agreement ``Advances on Fat Points and Symbolic Powers,'' Number H98230-11-1-0139.
The United States Government is authorized to reproduce and distribute reprints
notwithstanding any copyright notice.
The third author acknowledges the support provided by NSERC.



\section{Preliminaries}\label{prelims}
Let $I\subsetneq R=k[\pr^N]=k[x_0,\ldots,x_N]$ be a homogeneous ideal. Then $I$ has a
homogeneous primary decomposition, i.e., a primary decomposition
$I=\bigcap_i Q_i$ where each $\sqrt{Q_i}$ is a homogeneous
prime ideal, and $Q_i$ is homogeneous and $\sqrt{Q_i}$-primary 
\cite[Theorem 9, p. 153]{ZS}.
We define the $m$-th {\it symbolic power} of $I$ to be the ideal $I^{(m)}=\bigcap_j P_{i_j}$,
where $I^m=\bigcap_i P_i$ is a homogeneous primary decomposition,
and the intersection $\bigcap_j P_{i_j}$ is over all primary components $P_i$ such that
$\sqrt{P_i}$ is contained in an associated prime of $I$.
Thus $I^{(m)}=\bigcap_{P\in\operatorname{Ass}(I)}(I^mR_P\cap R)$, where $R_P$ is the localization at $P$ and
the intersection is taken over the associated primes $P$ of $I$.
In particular, we see that $I^{(1)}=I$ and that $I^m\subseteq I^{(m)}$.

In Corollary \ref{introcor1}, we are interested in the ideal $I$ of a scheme $X$
which is a union of $s$ disjoint lines $L_i\subset\pr^N$ with $N \geq 3$. 
The $m$-th symbolic power of $I$ in this case is $I^{(m)} = \bigcap_{i=1}^s I(L_i)^m$.

If $I$ is a complete intersection (i.e., generated by a regular sequence),
such as is the case for the ideal of a single line in $\pr^N$ for $N>1$,
then $I^{(r)}=I^r$ for all $r\geq 0$ (see \cite[Lemma 5, Appendix 6]{ZS}), 
hence $\gamma(I)=1$, and, as noted in the discussion before Theorem \ref{mainThm},
$\rho_a(I)=\rho'_a(I)=\rho(I)=1$.
More generally, if $I$ is the ideal of a smooth subscheme $X\subset\pr^N$, then 
$X$ is locally a complete intersection, but powers of ideals which are
complete intersections are unmixed (see \cite[Appendix 6]{ZS}) so the only possible
associated primes for powers of $I$ are the minimal primes 
and the irrelevant prime $M=(x_0,\ldots,x_N)$. 
But $M$ is never an associated prime for the ideal of a subscheme,
so $I^{(m)}$ is obtained from $I^m$ by removing the $M$-primary component
(if any); i.e., $I^{(m)}$ is the {\em saturation} $\operatorname{sat}(I)$ of $I$.
In particular, the degree $t$ homogeneous components $(I^m)_t$ of $I^m$ and 
$(I^{(m)})_t$ of $I^{(m)}$ agree for $t\gg0$. The least $s$ for which 
$(I^m)_t=(I^{(m)})_t$ for all $t\geq s$ is called the \emph{saturation degree} of $I$,
denoted $\operatorname{satdeg}(I^m)$.

It is known that $\operatorname{satdeg}(I^r)\leq \operatorname{reg}(I^r)$,
where $\operatorname{reg}(I^r)$ denotes the (Castelnuovo-Mumford) 
regularity of $I^r$; see \cite[Remark 1.3]{BS}.
Moreover, by \cite[Corollary 3, Proposition 4]{K}, the regularity of $I^r$ is
bounded above by a linear function $\lambda_Ir+c_I$ of $r$
and moreover $\lambda_I\leq\omega(I)\leq \operatorname{reg}(I)$.
(Unfortunately the constant term $c_I$ may be
positive so we know only that the regularity is bounded above by
$\lambda_I r+c_I$ for some $c_I$; we do not know that $\lambda_Ir$ is
an upper bound.)

The bounds given in Theorem \ref{mainThm} involve the asymptotic quantity
$\gamma(I)$, defined for a homogeneous ideal $(0)\neq I\subseteq k[\pr^N]$ 
as $\lim_{m\to\infty}\alpha(I^{(m)})/m$.
For the fact that this limit exists, 
see \cite[Lemma 2.3.1]{BH}. 
It is also known that $\gamma(I)\leq \alpha(I^{(m)})/m$ for all $m\geq 1$.
(This is because $\gamma(I)=\lim_{t\to\infty}\alpha(I^{(tm)})/(tm)$,
but $(I^{(m)})^t\subseteq I^{(mt)}$, hence $\alpha(I^{(tm)})\leq \alpha((I^{(m)})^t)=t\alpha(I^{(m)})$, so
$\alpha(I^{(tm)})/(tm)\leq \alpha(I^{(m)})/m$.)
There are also lower bounds for $\gamma(I)$; indeed, 
$\alpha(I^{(m)})/(m+N-1)\leq \gamma(I)$ (see \cite[Section 4.2]{HaHu}).
In fact, the proof given there works with $N$ replaced by
the largest height $h_I$ of the associated primes of $I$; i.e.,
$$\frac{\alpha(I^{(m)})}{m+h_I-1}\leq \gamma(I).\eqno{(2.1)}$$
Note that $h_I=N-1$ for the case of a radical ideal defining lines in $\pr^N$.
Thus while $\gamma(I)$ is hard to compute, one can estimate
$\gamma(I)$ arbitrarily accurately by computing values of $\alpha(I^{(m)})$.

In the statement of Theorem \ref{mainThm} we divide by $\gamma(I)$.
This is allowed since $\gamma(I)>0$. In fact, $\gamma(I)\geq1$ when $I$ is not
$(0)$ and not $(1)$; see \cite[Lemma 8.2.2]{refPSC}.

Given homogeneous ideals $(0)\neq I\subseteq J$ we clearly have
$\alpha(I)\geq\alpha(J)$. This, and the easy fact that $\alpha(I^r)=r\alpha(I)$, give
a useful criterion for showing $I^{(m)}\not\subseteq I^r$, namely,
if $\alpha(I^{(m)})<r\alpha(I)$, then $I^{(m)}\not\subseteq I^r$.

\renewcommand{\thetheorem}{\thesection.\arabic{theorem}}
\setcounter{theorem}{0}

\section{Proof of Theorem \ref{mainThm} and Corollary \ref{introcor1}}\label{tmt}

Recall that $M=(x_0,\ldots,x_N)$ is the irrelevant ideal in $k[\pr^N]$.

\begin{proof}[Proof of Theorem \ref{mainThm}] 
(1) It is clear from the definitions that $\rho_a(I)\leq\rho'_a(I)\leq\rho(I)$.
As noted in section \ref{prelims} above, $1\leq \gamma(I)\leq \alpha(I)$. Thus
$\alpha(I)/\gamma(I)$ makes sense and we have
$1\leq \alpha(I)/\gamma(I)$.
The bound $\alpha(I)/\gamma(I)\leq\rho(I)$
is \cite[Lemma 2.3.2(b)]{BH}. The proof, which we now recall,
actually shows $\alpha(I)/\gamma(I)\leq\rho_a(I)$.
It is enough to show that $m/r<\alpha(I)/\gamma(I)$
implies $I^{(mt)}\not\subseteq I^{rt}$ for $t\gg0$ and hence that
$m/r\le\rho_a(I)$. But $m/r<\alpha(I)/\gamma(I)$ implies
$m\gamma(I)<r\alpha(I)$, so for all $t\gg0$ we have
$m\alpha(I^{(mt)})/(mt)< r\alpha(I)$, hence $\alpha(I^{(mt)})< rt\alpha(I)$
so $I^{(mt)}\not\subseteq I^{rt}$.
We also have $I^{(h_Ir)}\subseteq I^r$ for all $r$ by \cite{HH}.
But $I^{(m)}\subseteq I^{(h_Ir)}$ if $m\geq h_Ir$, so $h_I\geq \rho(I)$.
Moreover, $h_I\leq N$ except in the case that
$M$ is an associated prime for $I$. But in this case $I^{(m)}=I^m$
follows from the definition, so $\rho(I)=1$. Thus $\rho(I)\leq N$
and hence $\rho(I)\leq h$.

(2) Note that $\omega(I)\geq\alpha(I)$, so $\omega(I)/\gamma(I)\geq\alpha(I)/\gamma(I)\geq 1$.
Thus, if $m/r>\omega(I)/\gamma(I)$, then
$m\gamma(I)>r\omega(I)$ and $m\geq r$. This means for any fixed constant $c$,
we have $mt\gamma(I) >rt\omega(I)+c$ for $t \gg 0$. So,
in particular, for $t\gg0$ we have
$$\alpha(I^{(mt)})\geq mt\gamma(I)>rt\omega(I)+c_I\geq
\operatorname{reg}(I^{rt})\geq \operatorname{satdeg}(I^{rt}).$$
Thus $(I^{(mt)})_l=0$ when $l<\operatorname{satdeg}(I^{rt})$
(since $\operatorname{satdeg}(I^{rt})<\alpha(I^{(mt)})$), but
$I^{(mt)}\subseteq I^{(rt)}$ (since $m\geq r$) so
$(I^{(mt)})_l\subseteq (I^{(rt)})_l=(I^{rt})_l$ when $l\geq 
\operatorname{satdeg}(I^{rt})$; i.e., we have $(I^{(mt)})_l\subseteq (I^{rt})_l$ for all
$l$ (and hence $I^{(mt)}\subseteq I^{rt}$)
for all $t\gg0$. Thus, as noted in Remark \ref{introrem}, we cannot have $m/r<\rho_a(I)$,
so $\rho_a(I)\leq m/r$ whenever $m/r>\omega(I)/\gamma(I)$; i.e.,
$\rho_a(I)\le \omega(I)/\gamma(I)$.
As is well known, $\omega(I)\leq \operatorname{reg}(I)$, 
which finishes the proof of (2).

(3) First note that if $m,r\geq 1$ and if there is an integer $s\geq 0$
such that $m\geq (s+1)c$ and $r\leq (s+1)b$, then
$$I^{(m)}\subseteq I^{((s+1)c)}= (I^{(c)})^{s+1}\subseteq (I^{b})^{s+1}=I^{(s+1)b}\subseteq I^r.$$
We now show that if $m,r\geq 1$ satisfies $r\leq mb/c-b$, then 
$I^{(m)}\subseteq I^r$. This is because we can take $s$ to be the largest
integer such that $sb\leq r$, hence $r<(s+1)b$. But now $m\geq (c/b)(r+b)\geq (c/b)(sb+b)=c(s+1)$.

In particular, if $m,r\geq t$ but $I^{(m)}\not\subseteq I^r$, then
$r>mb/c-b$ so $c/b > (m/r)-(c/r)$ hence $(c/b) + (c/t) > m/r$. Therefore,
$\rho(I,t)\leq (c/b) + (c/t)$, so $\rho'_a(I)=\limsup_t \rho(I,t)\leq \limsup_t (c/b) + (c/t)=c/b$.
\end{proof}

\begin{proof}[Proof of Corollary \ref{introcor1}]
We need some facts about ideals of lines in $\pr^N$.
For the ideal $I$ of the union $X$ of $s$ general lines in $\pr^N$ for $N\geq 3$, \cite{HaHi} shows that
$\dim I_t = \max(0,\binom{t+N}{N}-s(t+1))$. 
Thus, if $s$ and $t$ are such that $\binom{t+N}{N}=s(t+1)$, then $\dim I_t =0$, but
$\binom{t+1+N}{N}-s(t+2)=\binom{t+N}{N}+\binom{t+N}{N-1}-s(t+2)=
\binom{t+N}{N}+\binom{t+N}{N}\frac{N}{t+1}-s(t+2)=s(t+1)+sN-s(t+2)=s(N-1)>0$, so
$\dim I_{t+1}>0$ and hence
$\alpha(I)=t+1$. We claim that $\operatorname{reg}(I)=t+1$ also.
So suppose $t\geq0$ and let
$\mathcal I$ be the sheafification of $I$. We have
$$0\to {\mathcal I}(i)\to {\mathcal O}_{\pr^N}(i)\to {\mathcal O}_{X}(i)\to 0.$$
By \cite{HaHi}, this is surjective on global sections if 
$i\geq t$ and ${\Gamma(\mathcal O}_{\pr^N}(i))\to \Gamma({\mathcal O}_{X}(i))$ is injective otherwise. Now by taking the cohomology of the sheaf sequence above,
it is easy to see that $h^j(\pr^N,{\mathcal I}(i-j))=0$ for all $i\geq t+1$ and $j\geq 1$
but that $h^2(\pr^N,{\mathcal I}(t-2))>0$ when $t=0$, and
that $h^1(\pr^N,{\mathcal I}(t-1))>0$ when $t>0$, for the latter
using the fact that 
$\binom{t-1+N}{N}-st=\binom{t+N}{N}-\binom{t-1+N}{N-1}-s(t+1)+s=
s-\binom{t-1+N}{N-1}=s-\binom{t+N}{N}\frac{N}{t+N}=s-s(t+1)\frac{N}{t+N}<0$.
Thus $\operatorname{reg}(I)=t+1$ by \cite[Exercise 20.20]{refE}.
Since $\alpha(I)=t+1=\operatorname{reg}(I)$,
the result now follows by Theorem \ref{mainThm}.
\end{proof}

\begin{remark}
Even when the bounds given in Theorem \ref{mainThm} do not directly give exact values,
they can be informative when combined with computational evidence. For example, consider the ideal of
$s=4$ general lines in $\pr^3$. In this case $\alpha(I)=3$, $\operatorname{reg}(I)=4$,
and, by \cite{GHVT}, $\gamma(I)=8/3$.
Thus $9/8\leq\rho_a(I)\leq 3/2$ by Theorem \ref{mainThm} (1) and (2).
But computational evidence
using {\it Macaulay2} suggests that $I^{(3t)} =
(I^{(3)})^t$ for all $t > 0$ and also that $I^{(9)}\subseteq I^8$.
This would imply by Theorem \ref{mainThm} (3) that
$\rho'_a(I)\leq 9/8$, and hence $\rho_a(I)=\rho'_a(I)=9/8$.
In fact, evidence from {\it Macaulay2} also suggests that $I^{(3t+i)}
= (I^{(3)})^tI^i$ for $i=1,2$ and all $t > 0$ and that
$I^{(3)}\subseteq I^2$ and $I^{(6)}\subseteq I^5$. Given that
these hold, one can prove that $I^{(m)}\subseteq I^r$ whenever
$m/r\geq 9/8$, which by definition implies that $\rho(I)\leq 9/8$
and thus that $\rho_a(I)=\rho'_a(I)=\rho(I)=9/8$. The proof is to
check cases modulo 3. For example, given any $m\geq1$, we can
write $m=3t+i$, for some $t$ and $0\leq i\leq2$, and we can write
$t=3j+q$ for some $0\leq q\leq 2$. Then
$I^{(m)}=I^{(9j+3q+i)}=(I^{(3)})^{3j+q}I^i=(I^{(9)})^j(I^{(3)})^qI^i$.
Also note that $m/r\geq 9/8$ is equivalent to
$8j+8q/3+8i/9=8m/9\geq r$. If $q=0$, then $I^{(9)}\subseteq I^8$
implies $I^{(m)}=(I^{(9)})^jI^i\subseteq I^{8j}I^i$, but
$I^{8j}I^i\subseteq I^r$, since $8j+i\geq 8j+8q/3+8i/9\geq r$. If
$q=1$, then $I^{(9)}\subseteq I^8$ and $I^{(3)}\subseteq I^2$
imply $I^{(m)}=(I^{(9)})^j(I^{(3)})^qI^i\subseteq I^{8j}I^2I^i$,
but $I^{8j}I^2I^i\subseteq I^r$ since
$8j+8/3+8i/9=8j+8q/3+8i/9\geq r$ implies
$8j+2+i=\lfloor8j+8/3+8i/9\rfloor\geq r$. If $q=2$, then
$I^{(9)}\subseteq I^8$ and $I^{(6)}\subseteq I^5$ imply
$I^{(m)}=(I^{(9)})^j(I^{(3)})^qI^i\subseteq I^{8j}I^5I^i$, but
$I^{8j}I^5I^i\subseteq I^r$ since
$8j+5+i=\lfloor8j+8q/3+8i/9\rfloor\geq r$.
\end{remark}

\renewcommand{\thetheorem}{\thesection.\arabic{theorem}}
\setcounter{theorem}{0}

\section{Proof of Theorem \ref{introthm2}}\label{monomialcases}

\begin{proof}[Proof of Theorem \ref{introthm2}]
If $s=1$, then $I$ is a complete intersection so 
$\gamma(I)=\rho_a(I)=\rho'_a(I)=\rho(I)=1=\max(1,2(s-1)/2)$
as noted above. So hereafter we assume $s>1$.
We begin by computing $\gamma(I)$ for $s=(N+1)/2$ general lines in $\pr^N$
in the case that $N$ is odd (hence $s\geq 2$). By choosing two points on each line, we obtain $N+1$
points of $\pr^N$ which after a change of coordinates we may assume are the coordinate
vertices. The $s$ lines, $L_1,\cdots,L_s$, now become disjoint coordinate lines. 
After reindexing, we may assume $I(L_i)$ is generated by the coordinate variables
$\{x_j: 0\leq j\leq N+1, j\neq 2(i-1),2i-1\}$. 
Since $I^{(m)}=I(L_1)^m\cap\cdots\cap I(L_s)^m$,
we see that $I^{(m)}$ is a monomial ideal, and a monomial $\mu=x_0^{m_0}\cdots x_N^{m_N}$
is in $I^{(m)}$ if and only if $\mu\in I(L_j)^m$ for all $j$.
To get a useful criterion for $\mu\in I(L_j)^m$, 
let $d_i(\mu)=m_{2(i-1)}+m_{2i-1}$ for $1\leq i\leq s$,
and let $d(\mu)=\deg(\mu)=m_0+\cdots+m_N$. 
Clearly $\mu\in I(L_j)^m$ holds if and only if the degree of $\mu$ in the 
variables generating $I(L_i)$
is at least $m$; i.e., if and only if $d(\mu)\geq m+d_i(\mu)$. 
Thus $\mu\in I^{(m)}$ if and only if $d(\mu)\geq m+d_i(\mu)$ holds for all $1\leq i\leq s$.
But $d(\mu)=d_1(\mu)+\cdots+d_s(\mu)$, so summing $d(\mu)\geq m+d_i(\mu)$ over $i$
we conclude that $\mu\in I^{(m)}$ implies $sd(\mu)\geq sm+d(\mu)$ or $d(\mu)\geq sm/(s-1)$.
Thus $\alpha(I^{(m)})\geq sm/(s-1)=(N+1)m/(N-1)$, hence $\gamma(I)\geq (N+1)/(N-1)$. 

To show in fact that $\gamma(I)=(N+1)/(N-1)$, consider the case that $m$ is a multiple of $s-1$,
so let $m=\lambda(s-1)$. We claim that $\alpha(I^{(m)})=\lambda s$.  We already know
$\alpha(I^{(m)})\geq (N+1)m/(N-1)=\lambda s$, so it suffices
to find some element of degree $\lambda s$ in $I^{(m)}$.   Consider $\mu=x_0^{m_0}\cdots x_N^{m_N}$
where $m_0=\cdots =m_N=\lambda/2$ for any even $\lambda$. Then 
$d(\mu)=\lambda s=\lambda(s-1)+\lambda=m+d_i(\mu)$ holds for all $i$, so
$\mu\in I^{(m)}$. Thus $\alpha(I^{(m)})= \lambda s$. Taking the limit of
$\alpha(I^{(m)})/m=\lambda s/(\lambda (s-1))$ as $\lambda\to\infty$ gives
$\gamma(I)=s/(s-1)=(N+1)/(N-1)$, as claimed. Moreover,
since $s=(N+1)/2$, we have $\binom{t+N}{N}=s(t+1)$ for $t=1$,
so, by Corollary \ref{introcor1}, $\rho_a(I)=(t+1)/\gamma(I)=2(N-1)/(N+1)=2(s-1)/s=\max(1,2(s-1)/s)$.

We now show that we also have $\rho(I)=2(s-1)/s$, and hence by Theorem \ref{mainThm}(1) that
$\rho'_a(I)=2(s-1)/s$. Consider the homomorphism 
$\phi:k[\pr^N]=k[x_0,\ldots,x_N]\to k[y_0,\ldots,y_{s-1}]=k[\pr^{s-1}]$
of polynomial rings where $y_0=\phi(x_0)=\phi(x_1)$, $y_1=\phi(x_2)=\phi(x_3)$, etc.
Note that $\phi(I(L_i))=J(p_i)$ where $p_0,\ldots,p_{s-1}$ are the coordinate vertices of $\pr^{s-1}$.
For any monomial $\mu'=y_0^{m_0}\cdots y_{s-1}^{m_{s-1}}\in k[\pr^{s-1}]$, 
define $d'(\mu')=\deg(\mu')=m_0+\cdots+m_{s-1}$ and define $d'_i(\mu')=m_i$.
Thus for any monomial $\mu\in k[\pr^N]$ we have $d'(\phi(\mu))=d(\mu)$
and $d'_i(\phi(\mu))=d_i(\mu)$. 
Therefore, if we set $J=\phi(I)$, then for any monomial $\mu\in k[\pr^N]$,
we have $d'(\phi(\mu))\geq m+d'_i(\phi(\mu))$ if and only if
$d(\mu)\geq m+d_i(\mu)$, so $\mu\in I^{(m)}$ if and only if $\phi(\mu)\in 
J(p_0)^m\cap\cdots\cap J(p_{s-1})^m$. In particular, $J=J(p_0)\cap\cdots\cap J(p_{s-1})$,
and $\phi(I^{(m)})=J^{(m)}$. 

We now see that $I^{(m)}\subseteq I^r$ implies $J^{(m)}=\phi(I^{(m)})\subseteq \phi(I^r)=\phi(I)^r=J^r$.
For the converse, assume $J^{(m)}\subseteq J^r$ and consider a monomial $\mu\in I^{(m)}$.
Then $\phi(\mu)\in J^{(m)}\subseteq J^r$ and $\phi(\mu)$ is a monomial, so we can factor 
$\phi(\mu)$ as $\phi(\mu)=\mu'_1\cdots\mu'_r$ with each $\mu'_i$ being a monomial in $J$.
For each $i$, there is a monomial $\mu_i\in k[x_0,\ldots,x_N]$ with $\phi(\mu_i)=\mu'_i$,
and since $\phi(\mu_i)\in J^{(1)}=J$, we have $\mu_i\in I^{(1)}=I$; i.e., $\mu\in I^r$.
Thus $I^{(m)}\subseteq I^r$ if and only if $J^{(m)}\subseteq J^r$, so
$\rho(I)=\rho(J)$, but $\rho(J)=2(s-1)/s$ by \cite[Theorem 2.4.3(a)]{BH}.

Now consider the case that $s<(N+1)/2$. Let $n=2s-1$. A monomial $\mu=x_0^{m_0}\cdots x_N^{m_N}$
factors as $\mu=\mu_1\mu_2$ where $\mu_1=x_0^{m_0}\cdots x_n^{m_n}$
and $\mu_2=x_{n+1}^{m_{n+1}}\cdots x_N^{m_N}$. Let $\delta = \delta(\mu)=m_{n+1}+\cdots+m_N$.
Then $\mu\in I^{(m)}$ if and only if $\mu_1\in I^{(m-\delta)}$, where we regard
nonpositive powers or symbolic powers as denoting the unit ideal. Similarly,
$\mu\in I^r$ if and only if $\mu_1\in I^{r-\delta}$.
Denote by $J$ the ideal of the $s$ lines regarded as being in $\pr^n$,
where $k[\pr^n]=k[x_0,\ldots,x_n]\subset k[x_0,\ldots,x_N]$.
Then $\mu_1\in I^{(m)}$ if and only if $\mu_1\in J^{(m)}$
and $\mu_1\in I^r$ if and only if $\mu_1\in J^r$.
Also, the $s$ lines in $\pr^n$ have $\rho_a(J)=\rho(J)=2(s-1)/s$ by the 
previously considered case. Now, for any monomial $\mu'\in J^{(m)}\setminus J^r$, we have
$\mu'\in I^{(m)}\setminus I^r$, so $\rho(I)\geq \rho(J)$ and $\rho_a(I)\geq \rho_a(J)$.
On the other hand, say $m$ and $r$ are such that $I^{(m)}\not\subseteq I^r$.
Then there is a monomial $\mu\in I^{(m)}\setminus I^r$ (hence $\delta=\delta(\mu)<r$,
since $\delta(\mu)\geq r$ implies $\mu\in I^r$). If $m<r$, then
$J^{(m)}\not\subseteq J^r$, so $m/r\leq\rho(J)$.
Now assume $m\geq r$. Since $\mu_1\in I^{(m-\delta)}\setminus I^{r-\delta}$,
we see $\mu_1\in J^{(m-\delta)}\setminus J^{r-\delta}$, 
so $m/r\leq(m-\delta)/(r-\delta)\leq \rho(J)$. Thus $m/r\leq\rho(J)$ whenever
we have $I^{(m)}\not\subseteq I^r$, so we conclude $\rho(I)\leq \rho(J)$ hence $\rho_a(I)\leq\rho_a(I)\leq\rho(I)\leq\rho(J)=\rho_a(J)$.
Thus $\rho_a(I)=\rho(I)=\rho(J)=\rho_a(J)$.

Finally, if $s<(N+1)/2$ ($N$ not necessarily odd), then the $s$ lines are contained in a hyperplane,
in particular, in $x_N=0$), so $\alpha(I^{(m)})=m$ and hence $\gamma(I)=1$.
\end{proof}

For the ideal $I$ of two disjoint lines in $\pr^N$, Theorem \ref{introthm2}
shows that $\rho(I)=1$. 
An alternative way to see that $\rho(I)=\rho_a(I)=1$ for the ideal
$I$ of two skew lines in $\pr^N$, is to show that $I^{(m)}=I^m$ for all $m\geq1$.
To see this, let $V\subseteq\{x_0,\ldots,x_N\}$ be a nonempty subset of 
the coordinate variables of $k[\pr^N]$. 
Let $\{V_j\}_{j=1,\ldots,r}$ be a partition of $V$ into nonempty proper disjoint subsets.
Let $I_j=(V_j)$. Since each ideal $I_j$ is a complete intersection,
we have $(I_j)^{(m)}=(I_j)^m$ for all $m\geq 1$. Then we have the following result.

\begin{lemma}\label{disjointvars}
Let $m_1,\ldots,m_r$ be nonnegative integers, not all zero.
Then $I_1^{m_1}\cdots I_r^{m_r}=\bigcap_j (I_j^{m_j})$.
\end{lemma}

\begin{proof}
Clearly we have $I_1^{m_1}\cdots I_r^{m_r}\subseteq \bigcap_j (I_j^{m_j})$.
Both ideals are monomial ideals. The former is generated by the monomials of the form
$\mu_1\cdots\mu_r$ where $\mu_j$ is a monomial of degree $m_j$ in the variables $V_j$.
Thus it is enough for any monomial $\mu\in \bigcap_j (I_j^{m_j})$ to show that
$\mu$ is divisible by such a monomial $\mu_1\cdots\mu_r$.
But $\mu\in I_j^{m_j}$ for each $j$ and $I_j^{m_j}$ is generated by
monomials of degree $m_j$ in the variables $V_j$, and so there is such
a monomial $\mu_j$ that divides $\mu$. Since the elements $\mu_1,\ldots,\mu_r$
are pair-wise relatively prime, we see that $\mu_1\cdots\mu_r$ divides $\mu$.
\end{proof}

\begin{remark}\label{2lines}
We now show that $I^{(m)}=I^m$ for all $m\geq1$ whenever $I$ is the ideal
of two skew lines in $\pr^N$.
So let $L_1,L_2\subset \pr^N$ be skew lines.
If $N=3$, by an appropriate choice of coordinates, we may assume 
$I(L_1)=(x_0,x_1)$ and $I(L_2)=(x_2,x_3)$, and $I=I(L_1)\cap I(L_2)$.
By Lemma \ref{disjointvars} we have $I=I(L_1)I(L_2)$ and
$I^{(m)}=I(L_1)^m\cap I(L_2)^m=(I(L_1)I(L_2))^m=I^m$.
If $N>3$, we have
$I(L_1)=(x_0,x_1,x_4,\ldots,x_N)$
and $I(L_2)=(x_2,\ldots,x_N)$, and so $I=(x_0x_2,x_0x_3,x_1x_2,x_1x_3,x_4,\ldots,x_N)$.
Clearly, $I^m\subseteq I^{(m)}$ for all $m\geq 1$, so let $\mu\in I^{(m)}$ be a monomial.
We have a factorization $\mu=\mu_1\mu_2\mu_3$ where $\mu_1\in(x_0,x_1)$,
$\mu_2\in(x_2,x_3)$ and $\mu_3\in(x_4,\ldots,x_N)$. 
Let $\delta_i=\deg(\mu_i)$, $\delta=\deg(\mu)$.
If $\delta_3\geq m$, then $\mu_3\in I^m$, hence $\mu\in I^m$.
Assume $\delta_3<m$. Then $\mu\in I^{(m)}$ implies
$\delta_i\geq m-\delta_3$ for $i=1,2$. Thus
$\mu_1 \in (x_0,x_1)^{m-\delta_3}$ and $\mu_2 \in (x_2,x_3)^{m-\delta_3}$.
By Lemma \ref{disjointvars}, we have $\mu_1\mu_2 \in (x_0,x_1)^{m-\delta_3} \cap
(x_2,x_3)^{m-\delta_3} = [(x_0,x_1)(x_2,x_3)]^{m-\delta_3} = (x_0x_2,x_0x_3,x_1x_2,x_1x_3)^{m-\delta_3}
\subseteq I^{m-\delta_3}$, and
$\mu=\mu_1\mu_2\mu_3\in I^m$.
\end{remark}


\section{Computing $\gamma(I)$}\label{gamma}

Corollary \ref{introcor1} gives an exact answer for $\rho_a(I)$, but it is in terms of $\gamma(I)$,
which it is hard to determine in general. Let $I$ be the ideal of $s$ generic lines in $\pr^N$.
In cases such that $I^{(m)}=I^m$ for all $m\geq 1$, we have 
$\gamma(I)=\alpha(I)$. Thus $\gamma(I)=1$ if $s=1$, since 
$I$ is a complete intersection, so $I^{(m)}=I^m$ for all $m\geq 1$.
When $s=2$, then again $I^{(m)}=I^m$ for all $m\geq 1$ by Remark \ref{2lines},
so $\gamma(I)=\alpha(I)=2$ if $N=3$ and $\gamma(I)=\alpha(I)=1$ if $N>3$. 
If $2s\leq N+1$, then we know $\gamma(I)$ by Theorem \ref{introthm2}.
If $N=s=3$, then we will show in a separate paper \cite{GHVT} that
once again $I^{(m)}=I^m$ for all $m\geq 1$, so $\gamma(I)=\alpha(I)=2$, and,
by exploiting a connection between lines in $\pr^3$ and points in $\pr^1\times\pr^1$,
that $\gamma(I)=8/3$ if $s=4$. (It is not hard to see at least that $\gamma(I)\leq 8/3$.
If $s=4$, then $\alpha(I)=3$ and there is a unique
quadric through any three of the four lines, so taking each combination of three of the four
lines we see that the four quadrics give an element of $I^{(3)}$ of degree 8.
Thus $\alpha(I^{(3)})\leq 8$, so $\gamma(I)\leq \alpha(I^{(3)})/3=8/3$.)

As far as we know, $\gamma(I)$ is not known in any other cases, but one can still
say something. For example, say $N=3$ and $s=5$. 
Then $\alpha(I)=4$, so using the lower bound $\alpha(I^{(m)})/(m+h_I-1)\leq\gamma(I)$
we have $2\leq \gamma(I)$. Using the ten quadrics through 
combinations of three of the five lines gives 
$\alpha(I^{(6)})\leq 20$, so we get $\gamma(I)\leq \alpha(I^{(6)})/6\leq 20/6$.
(Experiments with {\it Macaulay2} suggest that $\alpha(I^{(6)})=20$.
If so, using the lower bound $\alpha(I^{(m)})/(m+h_I-1)\leq\gamma(I)$, 
we would have $20/7\leq\gamma(I)$. As an application,
this would imply by Corollary \ref{introcor1} that
$6/5\leq\rho_a(I)\leq 7/5$. In fact, {\it Macaulay2} calculations suggest that 
$\alpha(I^{(12)})=40$ which would give $40/13\leq \gamma(I)$ and 
$6/5\leq\rho_a(I)\leq 13/10$.)

There is also the question of what we might expect $\gamma(I)$ to be for the ideal $I$
of $s$ generic lines in $\pr^N$. The analogous question for ideals of generic points in $\pr^N$
is also still open in general but there are conjectures for what the answer should be. 
Nagata \cite{Na} posed a still open conjecture for ideals of points in $\pr^2$, in connection with his
solution of Hilbert's 14th problem. In our terminology we can paraphrase it as follows.

\begin{conjecture}[Nagata's Conjecture]
Let $I\subset k[\pr^2]$ be the ideal of
$s\geq 10$ generic points of $\pr^2$.
Then $\gamma(I)=\sqrt{s}$.
\end{conjecture}

A generalization has been given by Iarrobino \cite{I} and Evain \cite{Evain} for generic points in
$\pr^N$, which we paraphrase in the following way:

\begin{conjecture}
Let $I\subset k[\pr^N]$ be the ideal of
$s\gg0$ generic points of $\pr^N$.
Then $\gamma(I)=\sqrt[N]{s}$.
\end{conjecture}

The motivation for these conjectures is that if $I$ is the ideal of $s$ generic points $p_i$
in $\pr^N$, then it is easy to check that 
$\dim (I(p_i)^m)_t = \max(0,\binom{t+N}{N}-\binom{m+N-1}{N})$.
It follows that $\dim (I^{(m)})_t\geq \max(0,\binom{t+N}{N}-s\binom{m+N-1}{N})$.
Thus $\alpha(I^{(m)})$ is at most the least $t$ such that $\binom{t+N}{N}>s\binom{m+N-1}{N}$.
An asymptotic analysis now shows that $\gamma(I)\leq\sqrt[N]{s}$, and the naive hope is that
the other inequality also holds.
Since it's known that it can fail for certain exceptional cases,
the conjecture is actually that the other inequality holds as long as $s$ is big enough,
where how big depends on $N$. 

It is of interest to see what we get if we apply the same reasoning to the problem
of determining $\gamma(I)$ for ideals $I$ of generic lines in $\pr^N$.
Let $L$ be a line in $\pr^N$. We can regard $I(L)$ as being 
$I(L)=(x_2,\ldots,x_N)\subset k[\pr^N]$. We can determine
$\dim (I(L)^m)_t$ for each $t\geq 0$ by counting the number of monomials
in $I(L)^m$ of degree $t$. Of course, $\dim (I(L)^m)_t=0$ for $t<m$. 
For $t\geq m$ and $N=3$, the result is
$\dim (I(L)^m)_t = \binom{t+3}{3}-\big((t+2)-(2m+1)/3\big)\binom{m+1}{2}$.
(Here is the argument. The monomials $x_0^ax_1^bx_2^cx_3^d$
of degree $t$ not in $I(L)^m$ are those for which
$c+d<m$. There are $(t-i+1)(i+1)$ monomials of degree $t$
such that $c+d=i$. Thus the number of monomials of degree $t$
not in $I(L)^m$ is 
\[
\begin{split}
\sum_{0\leq i<m}(t-i+1)(i+1)&=\sum_{0\leq i<m}(t+2-(i+1))(i+1)\\
&=(t+2)\sum_{0\leq i<m}(i+1)-\sum_{0\leq i<m}(i+1)^2\\
&=(t+2)\binom{m+1}{2}-(2m+1)\binom{m+1}{2}/3.
\end{split}
\]
Subtracting this from $\binom{t+3}{3}$ thus gives $\dim (I(L)^m)_t$.)

Thus for the ideal $I$ of $s$ distinct lines in $\pr^3$ 
we have $\dim (I^{(m)})_t =0$ for $t<m$ and for $t\geq m$ we have
$$\dim (I^{(m)})_t \geq 
\max\Bigg(0,\binom{t+3}{3}-s\big((t+2)-(2m+1)/3\big)\binom{m+1}{2}\Bigg).\eqno{({}^{\circ})}$$

We can use this to get an upper bound on $\gamma(I)$ when $s$ is not 
too small. To do so, let
$\tau=g$ be the largest real root of $\tau^3-3s\tau+2s$. It is easy to see that 
$\sqrt{3s}-(3/4)<g<\sqrt{3s}$ as long as $s\geq 1$.

\begin{lemma}\label{gammalemma}
If $t/m>g$, then $\binom{it+3}{3}-s\big((it+2)-(2im+1)/3\big)\binom{im+1}{2}>0$
for $i\gg0$, while for $1\leq t/m<g$, 
$\binom{t+3}{3}-s\big((t+2)-(2m+1)/3\big)\binom{m+1}{2}<0$ if $s\geq 17$.
\end{lemma}

\begin{proof}
It is convenient to substitute $m\tau$ into
$\binom{it+3}{3}-s\big((it+2)-(2im+1)/3\big)\binom{im+1}{2}$ for $t$.
Doing so (and multiplying by 6 to clear denominators) gives the following polynomial in $i$:
$$i^3m^3(\tau^3-3s\tau+2s)+i^2m^2(6\tau^2-3s\tau-3s)+im(11\tau-5s)+6.$$
The condition $t/m>g$ is now equivalent to $\tau>g$,
which means the leading coefficient of the polynomial is positive,
hence for $i\gg0$ the polynomial is positive.

Now assume $1\leq t/m<g$. Substituting $\tau=mt$ into 
$\binom{t+3}{3}-s\big((t+2)-(2m+1)/3\big)\binom{m+1}{2}$ for $t$
gives 
$$m^3(\tau^3-3s\tau+2s)+m^2(6\tau^2-3s\tau-3s)+m(11\tau-5s)+6.$$
It's easy to check that $\tau^3-3s\tau+2s\leq 0$ for $1\leq \tau<g$.
Also, $6\tau^2-3s\tau-3s<0$ for $1\leq \tau<g$ if $s\geq 12$
(since $6\tau^2-3s\tau\leq 0$ for $0\leq\tau\leq s/2$, and 
$s/2\geq\sqrt{3s}\geq g$ for $s\geq 12$).
Finally, $11\tau-5s\leq 0$ for $\tau\leq 5s/11$, and $5s/11\geq \sqrt{3s}>g$
for $s\geq (11\sqrt{3}/5)^2=14.52$. To accommodate the constant
term, 6, of the polynomial, we actually want $11\tau-5s+6 \leq 0$,
which occurs for $\tau\leq (5s-6)/11$, and we have $(5s-6)/11\geq \sqrt{3s}$
if $s\geq 17$. Thus each term of 
$\binom{t+3}{3}-s\big((t+2)-(2m+1)/3\big)\binom{m+1}{2}$ is nonpositive and one is
negative if $1\leq t/m<g$ and $s\geq 17$. 
\end{proof}

\begin{corollary}
Let $I$ be the ideal of $s\geq 1$ distinct lines in $\pr^3$.
Then $\gamma(I)\leq g$, where $g$ is the largest real root of
$\tau^3-3s\tau+2s$.
In particular, $\gamma(I)\leq \sqrt{3s}$, so $\rho_a(I)\geq \alpha(I)/\sqrt{3s}$
for any $s$ distinct disjoint lines in $\pr^3$. 
\end{corollary}

\begin{proof}
From Lemma \ref{gammalemma} we see that
$\alpha(I^{(im)})\leq it$ for $i\gg 0$ whenever $t/m>g$.
Thus $\gamma(I)=\lim_{i\to\infty}\alpha(I^{(im)})/(im)\leq t/m$ whenever
$t/m>g$, hence $\gamma(I)\leq g$.
The bound on $\rho_a(I)$ follows by Theorem \ref{mainThm}(1).
\end{proof}

If the $s$ lines are generic and $s\gg0$,
we might hope that $({}^\circ)$ is an equality for all $m,t\geq 1$.
Certainly it is not an equality for all $s, m$ and $t$, since  
for the ideal $I$ of $s=3$ general lines in 
$\pr^3$ we know (as noted above) that $I^{(m)}=I^m$,
but $\dim(I)_{2}=1$ and hence $\dim(I^{(m)})_{2m}=\dim(I^m)_{2m}=1$. However, 
$\binom{t+3}{3}-s\big((t+2)-(2m+1)/3\big)\binom{m+1}{2}<0$
for $t=2m$ and $s=3$ when $m>1$.
But suppose for $s$ generic lines in $\pr^3$ for each $s\gg0$ it were true that
$({}^\circ)$ is an equality for all $m,t\geq 1$;
then, by Lemma \ref{gammalemma}, it would follow that
$\alpha(I^{(m)})\geq gm$ for $m\geq1$ and hence that
$\gamma(I)\geq g$. Since we know from above that
$\gamma(I)\leq g$, we would have $\gamma(I)=g$.

It is therefore tantalizing to make the following conjecture:

\begin{conjecture}
There is an integer $q$ such that, for the ideal $I$ of $s\geq q$
sufficiently general lines (say $s$ generic lines) in $\pr^3$,
$\gamma(I)$ is equal to the largest real root $\tau=g$ of $\tau^3-3s\tau+2s$.
\end{conjecture}


\end{document}